\documentclass[12pt,reqno]{article}
\usepackage{amssymb}
\usepackage{amsmath}
\usepackage{amsfonts}
\usepackage{amsthm}
\usepackage{centernot}
\usepackage{enumitem}
\newtheorem{corollary}{Corollary}

\newtheorem{theorem}{Theorem}

\newtheorem{remark}{Remark}
\begin{document}
\title{A Pépin-Type Characterization for Fermat Pseudoprimes}
\author{Paolo Starni}
\date{}

\maketitle
\begin{abstract}
\noindent
Pépin's primality test asserts that, for $n\ge2$,
\[
3^{(F_n-1)/2}\equiv -1 \pmod{F_n}
\]
if and only if $F_n$ is prime. We establish a natural analogue of Pépin's criterion for pseudoprimality. More precisely, we prove that, for \mbox{$n\ge5$,}
\[
3^{(F_n-1)/2}\equiv 1 \pmod{F_n}
\]
if and only if $F_n$ is pseudoprime to the base $3$. 
\end{abstract}

\medskip

\section{Introduction}

Let $F_n = 2^{2^n} + 1$, $n \ge 0$, denote the $n$th Fermat number.

A classical result due to Pépin states that, for $n \ge 2$,
\[
3^{(F_n-1)/2} \equiv -1 \pmod{F_n}
\quad \Longleftrightarrow \quad
F_n \text{ is prime}.
\]

\noindent
The base $3$ is one of the admissible bases; other examples include $5, 10, \dots$ (p.~84 in \cite{Ri}). We restrict to base $3$, although all the results obtained below remain valid for other admissible bases.

The aim of this note is to establish a pseudoprimality criterion for Fermat numbers that naturally complements Pépin's primality test.

More precisely, we prove that a Fermat number is pseudoprime to the base $3$ if and only if Pépin's congruence holds with $1$ in place of $-1$. Thus Pépin's criterion admits a complete analogue for base-$3$ pseudoprimality. As a further consequence, we show that the congruence
\[
3^{(F_n-1)/4}\not\equiv -1\pmod{F_n}
\]
has no solutions for every $n\ge 5$.

\section{The main result}
Since there are no composite Fermat numbers for $n<5$, we restrict attention to $n\ge 5$.
\begin{theorem}
Let $\gcd(a,F_n)=1$. If $F_n$ is pseudoprime to the base $a$, then
\[
a^{(F_n-1)/4} \equiv 1 \pmod{F_n}.
\]
\end{theorem}

\begin{proof}
Let 
\[
F_n=\prod_{i=1}^{k} p_i^{e_i}
\]
be the prime factorization of $F_n$. Since $F_n$ is pseudoprime to the base $a$,
\[
a^{F_n-1}\equiv 1 \pmod{p_i^{e_i}} \quad \text{ for all } i.
\]
Hence the order of $a$ modulo $p_i^{e_i}$ divides both 
\[
\varphi(p_i^{e_i})=p_i^{e_i-1}(p_i-1)
\]
and 
\[
F_n-1=2^{2^n}.
\]
Therefore
\[
\operatorname{ord}_{p_i^{e_i}}(a)\mid 2^{2^n}.
\]
Since any divisor of a power of two is itself a power of two, we have
\[
\operatorname{ord}_{p_i^{e_i}}(a)=2^{\alpha_i}, \quad \alpha_i\ge 0.
\]
The order modulo $F_n$ is the least common multiple of the orders modulo its prime power factors:
\[
\operatorname{ord}_{F_n}(a)=\operatorname{lcm}\bigl(\operatorname{ord}_{p_i^{e_i}}(a)\bigr)
=2^{\alpha^*}, \quad \alpha^*=\max\{\alpha_i\}.
\]

It remains to show that $\alpha^*\le 2^n-2$.  Suppose, for contradiction, that $\alpha^*=2^n-1$. Then there exists a prime divisor $p_s$ of $F_n$ such that
\[2^{2^n-1}\mid \varphi(p_s^{e_s}).\]
Since $p_s$ is odd, $p_s^{e_s-1}$ is odd and thus
 \[2^{2^n-1}\mid (p_s-1).\]
Hence
\[
p_s=h\cdot 2^{2^n-1}+1.
\]
By Lucas's theorem on prime divisors of Fermat numbers
(p. 376 in \cite{D}),
\[
p_s=k2^{n+2}+1.
\]
Moreover, $k$ cannot be equal to $1$ or to a power of $2$. Indeed, otherwise
\[p_s=2^m+1\]
for some integer $m$. Since $p_s$ is prime, a classical result implies that $m$ is a power of 2  (p.~18 in \cite{H}). Hence $p_s=F_r$. This is impossible because $p_s\mid F_n$ while distinct Fermat numbers are pairwise coprime (p.~17 in \cite{H}).

\noindent 
If $h=1$ or $h=2^t$, then $p_s$ would be of the form $2^m+1$, a contradiction.

\noindent
Therefore $h$ cannot be equal to $1$ or to a power of $2$. Hence $h$ must have an odd divisor $q>1$.
Since $q\ge 3$, we obtain
\[
p_s=h\cdot 2^{2^n-1}+1
\ge q\cdot 2^{2^n-1}+1
\ge 3\cdot 2^{2^n-1}+1
>2^{2^n}+1=F_n,
\]
which contradicts $p_s\mid F_n$.

Thus $\alpha^*\le 2^n-2$, and therefore
\[
\operatorname{ord}_{F_n}(a)\mid 2^{2^n-2}=\frac{F_n-1}{4}.
\]
Consequently,
\[
a^{(F_n-1)/4}\equiv 1 \pmod{F_n}.
\]

\end{proof}

The following consequence of Theorem~1 is weaker but closer in form to Pépin's criterion.
\begin{corollary}
If $F_n$ is pseudoprime to the base $3$, then
\[
3^{(F_n-1)/2} \equiv 1 \pmod{F_n}.
\]
\end{corollary}
\begin{proof}
This follows immediately from Theorem 1 after squaring.
\end{proof}
The converse also holds.
\begin{corollary}
If  
\[
3^{(F_n-1)/2} \equiv 1 \pmod{F_n},
\]
then $F_n$ is pseudoprime to the base $3$.
\end{corollary}
\begin{proof}
Since $(3^{(F_n-1)/2})^2=3^{F_n-1} \equiv 1 \pmod{F_n}$, $F_n$ is prime or pseudoprime to the base $3$. But it cannot be prime because $3^{(F_n-1)/2}\equiv 1 \pmod{F_n}$ contradicts Pépin's test.
\end{proof}
Combining Corollaries 1 and 2, we obtain the following criterion.
\begin{theorem}[Base-$3$ pseudoprimality criterion]
\[
F_n \text{ is pseudoprime to base } 3 
\;\;\Longleftrightarrow\;\;
3^{(F_n-1)/2} \equiv 1 \pmod{F_n}.
\]
\end{theorem}
The following remark records the possible values $3^{(F_n-1)/2}  \pmod{F_n}$.
\begin{remark}
The value of
\[
3^{(F_n-1)/2}\pmod{F_n}
\]
determines whether $F_n$ is prime ($-1$), pseudoprime to the base $3$ ($1$), or composite and not pseudoprime to the base $3$ ($\neq \pm1$). 
\end{remark}
\subsection*{A further consequence}
The following corollary establishes an unexpected congruence satisfied by every Fermat number with $n\geq5$.

\begin{corollary}
\[
3^{(F_n-1)/4}\not\equiv -1 \pmod{F_n}.
\]
\end{corollary}

\begin{proof}
Suppose, for contradiction, that
\[
3^{(F_n-1)/4}\equiv -1 \pmod{F_n}.
\]
Then
\[
3^{(F_n-1)/2}\equiv 1 \pmod{F_n}.
\]
By Theorem~2, $F_n$ is pseudoprime to the base $3$. Hence, by Theorem~1,
\[
3^{(F_n-1)/4}\equiv 1 \pmod{F_n},
\]
contradicting the assumption.
\end{proof}


\begin{thebibliography}{9}
\bibitem{D} L. E. Dickson, \emph{History of the Theory of Numbers}, Vol. 1, Dover, 2005.
\bibitem{H} G. H. Hardy and E. M. Wright, \emph{An Introduction to the Theory of Numbers}, 6th ed., Oxford, 2008.
\bibitem{Ri} P. Ribenboim, \emph{The New Book of Prime Number Records}, Springer, 2012.
\end{thebibliography}
\end{document}